\documentclass[11pt]{article}
\usepackage[margin=1.3in]{geometry}
\usepackage{amsmath,amssymb,amsthm,mathtools}
\usepackage{xcolor}
\usepackage{microtype}
\usepackage[colorlinks=true,linkcolor=blue!55!black,urlcolor=blue!55!black,citecolor=blue!55!black]{hyperref}
\newtheorem{theorem}{Theorem}[section]
\newtheorem{lemma}[theorem]{Lemma}
\newcommand{\R}{\mathbb R}
\newcommand{\LL}{\mathcal L}
\newcommand{\BV}{\mathrm{BV}}
\newcommand{\BD}{\mathrm{BD}}

\newcommand{\mres}{\!\mathbin{\vrule height 1.6ex depth 0pt width
		0.13ex\vrule height 0.13ex depth 0pt width 1.1ex}\!}

\title{Functions of bounded variation are not dense in the 
	 BD Norm}
\author{Panu Lahti\thanks{Email: \href{mailto:panulahti@amss.ac.cn}{\texttt{panulahti@amss.ac.cn}}.}
\quad and \quad Quoc-Hung Nguyen\thanks{Email: \href{mailto:qhnguyen@amss.ac.cn}{\texttt{qhnguyen@amss.ac.cn}}.}\\[0.5em]
\small Academy of Mathematics and Systems Science,\\
\small Chinese Academy of Sciences,\\
\small Beijing 100190, People's Republic of China}
\date{}
\begin{document}
\maketitle
\begin{abstract}
We prove that functions of bounded variation are not dense in the space
of functions of bounded deformation with respect to the 
BD norm.
Using alternating Riesz products and directional slicing,
 we construct on every nonempty open subset of $\R^n$  ($n\ge2$)
a continuous, compactly supported BD function that is at positive distance from BV.
\end{abstract}


\section{Introduction and main result}
\label{sec:negative}
Functions of bounded deformation (BD) describe displacements whose
linearized strain is a measure. The space was developed in connection
with variational models of plasticity; see Suquet~\cite{Suquet} and
Temam--Strang~\cite{TemamStrang}.
Ambrosio--Coscia--Dal Maso~\cite{ACDM} established the fine properties of BD functions,
including the decomposition of the symmetric derivative and the
structure of the jump set.
De Philippis--Rindler~\cite[Theorems~1.1 and~1.7]{DPR2016}
proved a structure theorem for $\mathcal A$-free measures and, in
particular, the symmetrized rank-one form of the polar of singular
strains in BD. They also characterized generalized Young measures
generated by symmetric gradients of BD fields~\cite{DPR2017}.
Their survey~\cite{DPR} presents the theory from the
viewpoint of linear PDE constraints. In the present paper, we study approximation by
BV functions (alternatively called BV fields) in the BD norm.


For an open set $\Omega\subset\R^n$ and $N\in\mathbb N$, the space $\BV(\Omega;\R^N)$
consists of $L^1$ functions whose distributional derivative $Dv$ is a finite
matrix-valued Radon measure.
We give precise definitions in Section \ref{sec:def}.
The space $\BD(\Omega)$ consists of functions
$u\in L^1(\Omega;\R^n)$ whose distributional symmetric derivative
$Eu=(Du+Du^T)/2$ is a finite matrix-valued Radon measure.
We use the Frobenius matrix norm to define the total variation $|Eu|$
and equip BD with the norm
\[
 \|u\|_{\BD(\Omega)}=\|u\|_{L^1(\Omega)}+|Eu|(\Omega).
\]
The inclusion $\BV(\Omega;\R^n)\subset\BD(\Omega)$ is continuous,
since $|Eu|(\Omega)\le|Du|(\Omega)$ for $u\in\BV(\Omega;\R^n)$.
Ornstein's non-inequality~\cite{Ornstein} is the classical obstruction
to an $L^1$ Korn inequality. It follows that in dimensions at least two, BV is a proper
subspace of BD. This raises a fundamental approximation problem in the
theory of functions of bounded deformation: is BV dense in BD with
respect to the BD norm? The present paper gives a negative answer.

BV theory contains results such as directional slicing~\cite{AFP}, the chain rules
of Vol'pert~\cite{Volpert} and Ambrosio--Dal Maso~\cite{ADM}, and
Alberti's rank-one theorem for the singular part of the
derivative~\cite{Alberti}. BD has its own fine structure
theory~\cite{ACDM,DPR2016,DPR}, but its norm controls only the symmetric
derivative. Ornstein's non-inequality~\cite{Ornstein} prevents control
of the full derivative by the symmetric one, so arguments that require
the full derivative cannot be transferred directly. Strong density
of BV in BD would offer a way to transfer arguments from BV to BD,
provided the relevant estimates are stable under this convergence.

A specific motivation comes from transport equations and regular
Lagrangian flows. Ambrosio--De Lellis--Mal\'y~\cite{ADLM2007} studied
chain rules for the divergence of BV-like vector fields and their
applications to transport equations. Ambrosio--Crippa--Maniglia~\cite{ACM2005}
developed normal trace theory for locally bounded BD fields and proved
renormalization results under SBD hypotheses. Nguyen~\cite{Nguyen2021}
established quantitative stability and
well-posedness for vector fields of the form
\[
 B^i=\sum_{j=1}^{m}K_j^i*b_j,
 \quad b_j\in L^1((0,T);\BV(\R^n)),
\]
under suitable assumptions on the singular kernels $K_j^i$, a growth
condition on $B$, and
$\operatorname{div}B\in L^1((0,T);L^\infty(\R^n))$.
We refer to this class as $K*\BV$. A natural question is whether the
theory extends to $K*\BD$, where the input is a vector-valued BD field.
Strong approximation of that input by BV fields would suggest a route
to such an extension.
The present theorem rules out a
general reduction based on strong BV approximation of arbitrary BD
inputs.

On bounded Lipschitz domains, smooth Sobolev functions are strictly dense
in BD; see
Barroso--Fonseca--Toader~\cite[Theorem~2.6]{BFT}. Strict convergence
means
\[
 u_j\to u\quad\text{in }L^1(\Omega),\quad
 Eu_j\overset{*}{\rightharpoondown}Eu,\quad
 |Eu_j|(\Omega)\to |Eu|(\Omega).
\]
Strong BD convergence, or convergence in the BD norm, requires convergence of the \emph{difference measure}
in total variation:
\[
 \|u_j-u\|_{L^1(\Omega)}+|Eu_j-Eu|(\Omega)\to 0.
\]
The area-strict approximation results in
\cite[Section~4, p.~409]{DPR} use a topology that is also weaker than the strong topology.

Several subclasses admit strong approximation by BV functions. Write
$\mathrm{LD}(\Omega)$ for the BD functions whose symmetric
derivatives are absolutely continuous with respect to Lebesgue measure,
and $\mathrm{SBD}(\Omega)$ for those with no Cantor part in their
symmetric derivatives. Dal Maso--Toader~\cite[Corollary~4.2]{DMT}
show that every SBD function on a bounded open set is the sum of an LD
function and an SBV function. On a bounded Lipschitz domain, this decomposition
and the strong smooth approximation of LD functions in
\cite[Theorem~2.6]{BFT} imply strong approximation of SBD functions by
BV functions. Crismale~\cite[Theorem~1.2]{Crismale} constructs approximants
with additional geometric structure and strong BD convergence;
see also~\cite{CrismaleLincei} for an account of these density results
and their applications to fracture energies.

Approximation and compactness have also been developed in the spaces
$\mathrm{GBD}$ and $\mathrm{GSBD}$, introduced by Dal Maso~\cite{DMGBD}
through slicing to treat displacements without an a priori integrability
bound. For $p>1$, the space $\mathrm{GSBD}^p$ consists of GSBD fields
whose approximate symmetric gradient belongs to $L^p$ and whose jump
set has finite $\mathcal H^{n-1}$ measure. On bounded Lipschitz domains,
Chambolle--Crismale~\cite{CCDensity} proved approximation in this space
by bounded SBV fields with regular
jump sets, with convergence in measure, strong $L^p$ convergence of the
approximate symmetric gradients, and control of the jump sets. This
construction yields approximation results for brittle fracture energies.

Chambolle--Crismale~\cite{CCCompactness} also proved compactness and lower
semicontinuity in GSBD under bounds on the Griffith energy, allowing
a set of finite perimeter on which the displacements diverge.
Under a uniform bound on the GBD slicing measures, their general
compactness theorem~\cite{CC2025} gives almost-everywhere convergence
along a subsequence after subtracting piecewise infinitesimal rigid
motions on a fixed Caccioppoli partition. More recently, they
characterized GBD by slicing conditions in finitely many
directions~\cite{CC2026}. These results provide tools for variational
problems in fracture mechanics. They do not, however, assert strong
BV approximation of arbitrary BD functions, whose symmetric derivatives
may have a Cantor part.

In the present paper, we use classical lacunary Riesz products~\cite{Riesz}.
Kazaniecki--Wojciechowski~\cite{KW2014,KW2024} developed finite
Riesz-product constructions for Ornstein-type non-inequalities;
alternating sums of successive increments also appear in
\cite[Section~4]{KW2024}.

The following is our main result.

\begin{theorem}[Non-density in the BD norm]\label{thm:nondensity}
Let $Q=(0,1)^2$. There exist $u\in C(\overline Q;\R^2)\cap\BD(Q)$,
functions $H_1,H_2\in C(\overline Q)$, and singular nonatomic probability
measures $\nu_1,\nu_2$ on $(0,1)$ such that, with
$\mu=\nu_1\otimes\nu_2$,
\begin{equation}\label{eq:main-strain}
 Eu=\begin{pmatrix}H_1&0\\0&H_2\end{pmatrix}\LL^2
   +\begin{pmatrix}0&1\\1&0\end{pmatrix}\mu,
\end{equation}
and
\begin{equation}\label{eq:main-gap}
 \inf_{v\in\BV(Q;\R^2)}|E(u-v)|(Q)\ge\sqrt2.
\end{equation}
Consequently, BV is not dense in BD in the norm topology of the latter. More generally,
for every $n\ge2$ and every nonempty open $\Omega\subset\R^n$, there is
$U\in C_c(\Omega;\R^n)\cap\BD(\Omega)$ whose distance
from $\BV(\Omega;\R^n)$ in the BD norm is positive.
\end{theorem}

We give definitions in Section \ref{sec:def},
we construct the measures in Section~\ref{sec:riesz-measures},
and we prove the theorem in Section~\ref{sec:alternating}.

\section{Definitions and some background theory}\label{sec:def}

Here we will give, in detail, definitions and background theory needed in the paper.
We mostly follow the monograph \cite{AFP}.

We work in $\R^n$, $n\in\mathbb{N}$.
For the space $\R^{n\times n}$ of $(n\times n)$-dimensional square matrices,
we use the Frobenius norm $|A|:=\sqrt{\sum_{j,k}(A_j^k)^2}$,
which is the Euclidean norm on $\R^{n^2}$.
The subspace of symmetric matrices is denoted by $\R^{n\times n}_{\textrm{sym}}$.
The Lebesgue measure in $\R^n$ is denoted by $\mathcal L^n$.

By $\Omega\subset \R^n$ we always mean an open set.
We denote by $\mathcal M^+(\Omega)$ the space of positive Radon measures on $\Omega$.  
For $\ell\in\mathbb{N}$, we denote by
$\mathcal M(\Omega;\R^{\ell})$ the Banach space of $\R^\ell$-valued
Radon measures $\mu$ on $\Omega$. The space is equipped with the
\emph{total variation norm} $|\mu|(\Omega)$, which is defined
with respect to the Euclidean norm on $\R^{\ell}$.
By the Riesz
representation theorem, $\mathcal M(\Omega;\R^{\ell})$ can be identified
with the dual space of $C_0(\Omega;\R^{\ell})$ through the duality
pairing
$\langle\phi,\mu\rangle\coloneqq \int_{\Omega} \phi\cdot d\mu\coloneqq \sum_{j=1}^{\ell} 
\int_{\Omega} \phi_j \,d\mu_j$.
Thus weak* convergence $\mu_i\overset{*}{\rightharpoondown}\mu$
in $\mathcal M(\Omega;\R^{\ell})$ means
$\langle\phi,\mu_i\rangle\to \langle\phi,\mu\rangle$ for all
$\phi\in C_0(\Omega;\R^{\ell})$.

For open sets $E\subset \R^{n-m}$, $F\subset \R^m$ with $m \in \{1,\ldots,n-1\}$,
a \emph{parametrized measure} $(\nu_{y})_{y \in E}$ is a mapping from $E$ to the set
$\mathcal M(F;\R^{\ell})$ of vector-valued Radon measures on $F$.
Given $\mu\in\mathcal M^+(E)$, it is said to be
\emph{weakly* $\mu$-measurable}, if $y\mapsto \nu_{y}(B)$ is
$\mu$-measurable for all Borel sets $B\subset F$.
Suppose that we additionally have
\[
\int_{E}|\nu_{y}|(F)\,d\mu(y)<\infty.
\]
Then we define the generalized product measure by
\begin{equation}\label{eq:product measure}
	\mu\otimes\nu_y(A)\coloneqq \int_E \left(\int_F \chi_A(y,t)\,d\nu_y(t)\right)\,d\mu(y)
\end{equation}
for every Borel set $A\subset E\times F$.

A function $u\in L^1(\Omega;\R^N)$, with $N\in\mathbb{N}$,
is a function of bounded variation,
denoted $u\in \BV(\Omega;\mathbb{R}^N)$, if its distributional derivative
is an $\R^{N\times n}$-valued Radon measure with finite total variation. This means that
there exists a (unique) Radon measure $Du\in \mathcal M(\Omega;\R^{N\times n})$
such that for all $\varphi\in C_c^1(\Omega)$, the integration-by-parts formula
\[
\int_{\Omega}u_k\partial_j\varphi\,dx
=-\int_{\Omega}\varphi\,d(Du)_j^k,\quad j=1,\ldots,n,\ k=1,\ldots,N
\]
holds.

We write the Radon-Nikodym decomposition of the variation measure of $u$ into the absolutely continuous and singular parts as $Du=D^a u+D^s u$. Furthermore, $D^su$ can be decomposed into the \emph{Cantor}
and \emph{jump} parts $D^c u$ and $D^j u$, of which $D^j u$ is carried on a set that is $\sigma$-finite
with respect to the $(n-1)$-dimensional Hausdorff measure, and $D^c u$ does not charge any set of ($\sigma$-)finite
$(n-1)$-Hausdorff measure. Thus we have the decomposition
\[
Du = D^a u + D^c u +D^j u.
\]
Analogous facts hold for BD functions, see \cite{ACDM}; apart from the Introduction, we will not specifically
use these decompositions in the paper, though it is useful to notice that the right-hand side of
\eqref{eq:main-strain} is the sum of an absolutely continuous and a Cantor part.


The space $\BD(\Omega)$ of functions of bounded deformation is the space of
functions $u\in L^1(\Omega;\R^n)$ whose distributional symmetrized derivative
is a finite Radon measure, that is, there exists a (unique) Radon measure
$Eu\in \mathcal M(\Omega;\R^{n\times n}_{\textrm{sym}})$ such that
for all $\varphi\in C_c^1(\Omega)$,
the integration-by-parts formula
\[
\frac{1}{2}\int_{\Omega}(u_j \partial_k \varphi+u_k\partial_j \varphi)\,dx
=-\int_{\Omega}\varphi \, d(Eu)_j^k,
\quad j,k=1,\ldots,n
\]
holds. Note that we do not write $\BD(\Omega;\R^n)$ because the target space necessarily
has the same dimension as the domain.

In $\R^n$, denote by $\pi_n\colon\R^n\to \R^{n-1}$ the orthogonal projection onto
$\R^{n-1}$:
for $x=(x_1,\ldots,x_n)\in\R^n$,
\begin{equation}\label{eq:projection}
\pi_n((x_1,\ldots,x_n)):=(x_1,\ldots,x_{n-1}).
\end{equation}
Similarly, define $\pi_j$, $j=1,\ldots,n-1$.
Denote the standard basis vectors by $e_k$, $k=1,\ldots,n$.
For any fixed $k\in\{1,\ldots,n\}$ --- for simplicity we can assume $k=n$ ---
for every $z\in\pi(\Omega)$ we denote the slices of $\Omega$ at $(z,0)$ in $e_n$-direction by
\[
\Omega_z\coloneqq \{t\in\R\colon (z,t e_n) \in \Omega\}.
\]
We also denote
$u_z(t)\coloneqq u(z,t)$ for $z\in\pi(\Omega)$ and $t\in \Omega_z$.
Given $u\in\BV(\Omega)$, for $\mathcal L^{n-1}$-almost
every $z\in\pi(\Omega)$, we have
$u_z\in \BV(\Omega_z)$ (see \cite[Theorem 3.103]{AFP}).
Recall \eqref{eq:product measure}.
Denoting by $D_l f$ the inner product between  $Df$ and $e_l$, we further have
\[
D_l f=\mathcal L^{n-1}\otimes D f_{z,l},
\]
see \cite[Theorem 3.107]{AFP}.
It follows that
\begin{equation}\label{eq:slice representation for total variation}
	|D_l f|=\mathcal L^{n-1}\otimes |D f_{z,l}|.
\end{equation}

In particular, for Borel sets $A,B\subset(0,1)$ of Lebesgue measure zero and $f\in\BV(Q)$
with $Q=(0,1)\times (0,1)$, we have
\begin{equation}\label{eq:BV-null-product}
 |D f|(A\times B)=0.
\end{equation}
The same is seen to be true for $f\in\BV(Q;\R^2)$ by considering $f$ componentwise.
The measure $\mu$ in our main theorem will be carried on such a set $A\times B$.

\section{The two singular Riesz measures}
\label{sec:riesz-measures}
We construct the singular product measure $\mu$ needed
in Theorem \ref{thm:nondensity}.

Let $\mathbb T=\R/\mathbb Z$, equipped with the Lebesgue measure
$\LL^1$. On $\mathbb T^2$, we consider the $2$-dimensional
Lebesgue measure $\LL^2=\LL^1\otimes\LL^1$.
Choose a parameter $a\in(0,1)$, which will remain fixed throughout the construction.
Set $S_0=0$, and recursively choose positive integers
$N_k$, and sums $S_{k-1}=\sum_{j<k}N_j$, choosing the $N_k$ so large that
\begin{align}
& N_k>4S_{k-1},\label{eq:dissociation}\\
& \frac{a(1+a)^{k-1}(\pi^{-1}+2S_{k-1})}
 {N_k-S_{k-1}}\le2^{-k}.\label{eq:growth}
\end{align}
For example, one can choose
$N_k$ to be an integer at least
\[
4S_{k-1}+2^ka(1+a)^{k-1}(\pi^{-1}+2S_{k-1})+1.
\]
Set $F_0=1$ and, for $k\ge1$, define
\begin{equation}\label{eq:riesz-product}
 F_k(x_1,x_2)=
 \prod_{\substack{1\le j\le k\\j\text{ odd}}}
 \bigl(1+a\cos(2\pi N_jx_1)\bigr)
 \prod_{\substack{1\le j\le k\\j\text{ even}}}
 \bigl(1+a\cos(2\pi N_jx_2)\bigr).
\end{equation}
Empty products are understood to equal one.
Each factor is bounded below by $1-a>0$.
The constant Fourier coefficient is one, so
$F_k\LL^2$ are probability measures on $\mathbb T^2$.

By a nonatomic measure $\nu$ we mean that
$\nu(\{t\})=0$ for all $t\in\mathbb T$.

\begin{lemma}\label{lem:riesz}
There exist nonatomic probability measures $\nu_1,\nu_2$ on $\mathbb T$,
both singular with respect to $\LL^1$, such that
\[
 F_k\LL^2\overset{*}{\rightharpoondown}\nu_1\otimes\nu_2=:\mu
 \quad\text{on }\mathbb T^2.
\]

Represent points of $\mathbb T$ by numbers in $[0,1)$; since $\nu_1(\{0\})=0=\nu_2(\{0\})$,
we may regard
$\nu_1$, $\nu_2$ as probability measures on $(0,1)$. There are Borel sets
$A,B\subset(0,1)$ such that
\[
 \LL^1(A)=\LL^1(B)=0,
 \quad \nu_1(A)=\nu_2(B)=\mu(A\times B)=1.
\]
\end{lemma}
\begin{proof}
For a finite measure $\nu$ on $\mathbb T$, the Fourier coefficients are
\[
 \widehat\nu(q)=\int_{\mathbb T}e^{-2\pi iqt}\,d\nu(t),
 \quad q\in\mathbb Z.
\]

\emph{Step 1: Uniqueness of signed frequency sums.}

We first prove a uniqueness result: suppose two finitely supported sequences
$(\varepsilon_j)$ and $(\varepsilon'_j)$ with entries in $\{-1,0,1\}$
satisfy $\sum_{j=1}^{\infty}\varepsilon_jN_j=\sum_{j=1}^{\infty}\varepsilon'_jN_j$.
We claim that $(\varepsilon_j)$ and $(\varepsilon'_j)$ necessarily coincide; suppose instead that they
are distinct, and let
$p=\max\{j:\varepsilon_j\ne\varepsilon'_j\}$.
Subtracting the sums cancels all terms with $j>p$. Taking absolute
values gives
\[
 N_p\le |\varepsilon_p-\varepsilon'_p|N_p
 =\left|\sum_{j<p}(\varepsilon_j-\varepsilon'_j)N_j\right|
 \le2S_{p-1}<N_p\quad\textrm{by }\eqref{eq:dissociation},
\]
a contradiction.
This proves the claim.
In particular, zero has only the all-zero representation.\\

%

\emph{Step 2: Construction and Fourier coefficients of $\nu_1$, $\nu_2$.}\\
Let $(M_j)_{j\ge1}$ denote either $(N_{2j-1})_{j\ge1}$ or
$(N_{2j})_{j\ge1}$, and set $R_0=1$ and
\[
 R_r(t)=\prod_{j=1}^r\bigl(1+a\cos(2\pi M_jt)\bigr),\quad r\in\mathbb{N}.
\]
We have
\[
 1+a\cos(2\pi M_jt)
 =1+\frac a2 e^{2\pi iM_jt}+\frac a2 e^{-2\pi iM_jt}.
\]
When we expand the product, we choose one of the three terms from each factor. For
$j=1,\ldots,r$, the number $\varepsilon_j\in\{-1,0,1\}$ records this
choice in the $j$-th factor: $\varepsilon_j=0$ selects $1$,
$\varepsilon_j=1$ selects $(a/2)e^{2\pi iM_jt}$, and
$\varepsilon_j=-1$ selects $(a/2)e^{-2\pi iM_jt}$.
Thus
the product expands as
\[
 R_r(t)=\sum_{\varepsilon_1=-1}^{1}\ldots
         \sum_{\varepsilon_r=-1}^{1}
 \left(\frac a2\right)^{\sum_{j=1}^r|\varepsilon_j|}
 e^{2\pi it\sum_{j=1}^r\varepsilon_jM_j}.
\]
For $q\in\mathbb Z$, we get the Fourier coefficient
\begin{align*}
 \widehat{R_r\LL^1}(q)
 &=\int_0^1 R_r(t)e^{-2\pi iqt}\,dt\\
 &=\sum_{\varepsilon_1=-1}^{1}\ldots
 \sum_{\varepsilon_r=-1}^{1}
   \left(\frac a2\right)^{\sum_{j=1}^r|\varepsilon_j|}
   \int_0^1 e^{2\pi it(\sum_{j=1}^r\varepsilon_jM_j-q)}\,dt.
\end{align*}
Only terms satisfying $\sum_{j=1}^r\varepsilon_jM_j=q$
contribute to this Fourier coefficient, and
by Step 1, there is at most one finitely supported sequence $(\varepsilon_j)$
satisfying the equality. 
Thus $\widehat{R_r\LL^1}(0)=1$
and for every fixed $q$, as we increase $r$, $\widehat{R_r\LL^1}(q)$ will eventually remain the same.
Since $R_r\ge0$, the measures $R_r\LL^1$
are probability measures. Weak compactness on $\mathbb T$, together
with the fact that Fourier coefficients determine a measure, gives
convergence of the full sequence to a probability measure $\nu$, with
\begin{equation}\label{eq:nu Fourier}
 \widehat\nu(q)=
 \begin{cases}
  (a/2)^{\sum_{j=1}^{\infty}|\varepsilon_j|},
   &q=\sum_{j=1}^{\infty}\varepsilon_jM_j
     \text{ with finitely supported }\varepsilon_j\in\{-1,0,1\},\\
  0,&\text{otherwise}.
 \end{cases}
\end{equation}
These coefficients are real. Write $\nu_1=\nu$ and $\nu_2=\nu$ for the cases of
$(N_{2j-1})_{j\ge1}$ and $(N_{2j})_{j\ge1}$, respectively.\\

\emph{Step 3: Convergence to the product measure.}\\
Write $R_r^{(1)}$ and $R_r^{(2)}$ for the products of Step 2 associated
with the odd and even subsequences. Then
\[
 F_k(x_1,x_2)=R_{\lceil k/2\rceil}^{(1)}(x_1)
          R_{\lfloor k/2\rfloor}^{(2)}(x_2).
\]
For $\varphi,\psi\in C(\mathbb T)$, Fubini's theorem therefore gives
\begin{align*}
 \int_{\mathbb T^2}\varphi(x_1)\psi(x_2)F_k(x_1,x_2)\,d\LL^1(x_1)\,d\LL^1(x_2)
 &=\left(\int_{\mathbb T}\varphi R_{\lceil k/2\rceil}^{(1)}\,d\LL^1\right)
   \left(\int_{\mathbb T}\psi R_{\lfloor k/2\rfloor}^{(2)}\,d\LL^1\right)\\
 &\to
   \left(\int_{\mathbb T}\varphi\,d\nu_1\right)
   \left(\int_{\mathbb T}\psi\,d\nu_2\right)\quad\textrm{as }k\to\infty.
\end{align*}
Finite sums of these product functions are uniformly dense in
$C(\mathbb T^2)$. Since all the measures have mass one, uniform approximation
extends the convergence to every continuous test function. Thus
$F_k\LL^2\overset{*}{\rightharpoondown}\nu_1\otimes\nu_2$.\\

\emph{Step 4: Singularity of $\nu_1$, $\nu_2$.}\\
Fix $\nu$ to be either $\nu_1$ or $\nu_2$, and consider the frequency sequence $(M_j)$
that is $(N_{2j-1})_{j\ge1}$ or $(N_{2j})_{j\ge1}$, respectively.
By equation \eqref{eq:nu Fourier}, and noting that e.g. $M_j$ is represented uniquely
(by Step 1) with the choice
$\varepsilon_j=1$ and other coefficients zero, we have
\begin{equation}\label{eq:widehat nu}
 \widehat\nu(M_j)=a/2,\qquad\widehat\nu(-M_j)=a/2,\qquad
 \widehat\nu(M_j\pm M_\ell)=a^2/4\quad(j\ne\ell).
\end{equation}
We also have
\begin{equation}\label{eq:2Mj}
\widehat\nu(2M_j)=0;
\end{equation}
indeed, note that $M_j$ is $N_p$ for some $p\in\mathbb{N}$, and then for any choice of
$(\varepsilon_j)$ we have
\[
\sum_{j=1}^{p}\varepsilon_j N_j\le N_p+S_{p-1}<2N_p\quad\textrm{by }\eqref{eq:dissociation},
\]
while if $\varepsilon_{\ell}\neq 0$ for a given $\ell>p$,  we get
\[
\Big|\sum_{j=1}^{\ell}\varepsilon_j N_j\Big|
\ge N_\ell-S_{\ell-1}>3S_{\ell-1}\ge3N_p.
\]
Thus $2M_j$ has no representation as a finite sum $\sum_{j}\varepsilon_j N_j$.

Consider the continuous functions
\[
 X_j(t)=\cos(2\pi M_jt),\quad
 Z_r(t)=\frac1r\sum_{j=1}^rX_j(t).
\]
On the second and third lines below, we use the identities $\cos^2 s=\tfrac 12 (1+\cos(2s))$
and $\cos s \cos t=\tfrac 12(\cos(s-t)+\cos(s+t))$.
Thus
\begin{align*}
 \int_{\mathbb T} X_j\,d\nu&=\int_{\mathbb T}\tfrac 12 (e^{2\pi iM_jt}+e^{-2\pi iM_jt}) \,d\nu(t)=
 \frac 12 (\widehat\nu(-M_j)+\widehat\nu(M_j))=a/2\quad \textrm{by }\eqref{eq:widehat nu};\\
 \int_{\mathbb T} X_j^2\,d\nu&
 =\tfrac12\int_{\mathbb T}(1+\cos(2\pi (2M_j)t)) \,d\nu(t)
 =\tfrac12\bigl(1+\tfrac 12 (\widehat\nu(-2M_j)+\widehat\nu(2M_j))
 \bigr)=\tfrac12\quad\textrm{by }\eqref{eq:2Mj};\\
 \int_{\mathbb T} X_jX_\ell\,d\nu
 &=
 \tfrac12\int_{\mathbb T}(\cos(2\pi (M_j-M_{\ell})t)+\cos(2\pi (M_j+M_{\ell})t))\,d\nu(t)\\
 &=\tfrac12\bigl(\widehat\nu(M_j-M_\ell)
                 +\widehat\nu(M_j+M_\ell)\bigr)=a^2/4
 \quad(j\ne\ell),\quad\textrm{by }\eqref{eq:widehat nu}.
\end{align*}
Hence 
\[
\int_{\mathbb T}(X_j-a/2)(X_{\ell}-a/2)\,d\nu=0\quad\textrm{for }j\ne\ell.
\]
It follows that
\[
 \int_{\mathbb T}|Z_r-a/2|^2\,d\nu
 =\frac{1}{r^2}\sum_{j=1}^r\int_{\mathbb T}(X_j-a/2)^2\,d\mathcal \nu
 =\frac{1/2-a^2/4}{r}.
\]
Moreover, since cosine functions with distinct frequencies are orthogonal with respect to the usual
Lebesgue measure, we have
\[
 \int_{\mathbb T}|Z_r|^2\,d\LL^1
 =\frac{1}{r^2}\sum_{j=1}^r \int_{\mathbb T}X_j^2\,d\LL^1=\frac1{2r}.
\]
For every $\varepsilon>0$, Chebyshev's inequality at the dyadic indices
$r=2^s$ gives
\begin{align*}
 \sum_{s=1}^{\infty}\nu\bigl(|Z_{2^s}-a/2|>\varepsilon\bigr)
 &\le\frac{1/2-a^2/4}{\varepsilon^2}\sum_{s=1}^{\infty}2^{-s}<\infty,\\
 \sum_{s=1}^{\infty}\LL^1\bigl(|Z_{2^s}|>\varepsilon\bigr)
 &\le\frac1{2\varepsilon^2}\sum_{s=1}^{\infty}2^{-s}<\infty.
\end{align*}
Apply the first Borel--Cantelli lemma for each threshold
$\varepsilon=1,1/2,1/3,\ldots$. By the lemma, for both measures we have that outside a null set,
every such threshold is exceeded only finitely often, and so
\[
 Z_{2^s}\to a/2\quad\nu\text{-a.e.},
 \quad Z_{2^s}\to 0\quad\LL^1\text{-a.e.}
\]
The set
\[
 E_\nu=\{t\in\mathbb T:\lim_{s\to\infty}Z_{2^s}(t)=a/2\}
\]
is Borel because the $Z_{2^s}$ are continuous. It satisfies
$\nu(E_\nu)=1$ and $\LL^1(E_\nu)=0$, since $a/2\ne0$. This proves
$\nu\perp\LL^1$.
Recalling that $\nu$ is either $\nu_1$ or $\nu_2$, we get the two $\mathcal L^1$-null sets
$E_1,E_2$ which carry $\nu_1,\nu_2$, respectively.\\

\emph{Step 5: Absence of atoms and null carriers in $(0,1)$.}\\
The two measures cannot have nonzero Fourier coefficients at the same
nonzero frequency $q\in\mathbb{Z}$; otherwise, that frequency would have two
finite sum representations $\sum_{j}\varepsilon_j N_j$,
one using odd $j$ and the other using even $j$, contradicting Step 1. The Fourier coefficients of the
convolution $\nu_1*\nu_2$ are products $\widehat{\nu_1}(q)\widehat{\nu_2}(q)$, so they equal
one at $q=0$ and vanish for other $q\in\mathbb{Z}$. These are exactly the coefficients
of the Lebesgue measure; hence $\nu_1*\nu_2=\LL^1$.

Suppose we had $\nu_1(\{t_0\})=c>0$ for some $t_0\in\mathbb T$. Since
$\nu_2(E_2)=1$, we would have
\[
\nu_1*\nu_2(t_0+E_2)\ge c.
\]
This set is Lebesgue-null, whereas the convolution equals the Lebesgue measure, giving a contradiction.
The same argument applies to $\nu_2$, so neither measure has atoms.

Represent torus points by numbers in $[0,1)$ and remove $0$ from the
Borel null sets $E_1,E_2$ to obtain $A,B\subset(0,1)$. Since neither
measure charges $0$, removing it loses no mass. Thus
$\nu_1(A)=\nu_2(B)=1$, and consequently $\mu(A\times B)=1$,
as required.
\end{proof}

\section{Proof of the non-density theorem}
\label{sec:alternating}
\begin{proof}[Proof of Theorem \ref{thm:nondensity}]

\emph{Step 1: Preliminary definitions and estimates.}\\
Let $i_k=1$ for odd $k$ and $i_k=2$ for even $k$.
Set
\[
 g_k=F_k-F_{k-1}
     =a\cos(2\pi N_kx_{i_k})F_{k-1}.
\]
Here $x_{i_k}$ is the coordinate of the new oscillation, and $x_{3-i_k}$ is the
other coordinate. For a trigonometric polynomial $p$ on
$\mathbb T^2$, write
\[
 \widehat p(\xi)=\int_{\mathbb T^2}p(x)e^{-2\pi i\xi\cdot x}\,d\LL^2(x),
 \quad
 \|p\|_{\mathcal A}=\sum_{\xi\in\mathbb Z^2}|\widehat p(\xi)|.
\]
The finite Fourier expansion and the convolution formula for a product give
\begin{equation}\label{eq:A norm}
 \|p\|_{L^{\infty}(\mathbb T^2)}\le\|p\|_{\mathcal A},
 \quad
 \|pq\|_{\mathcal A}\le\|p\|_{\mathcal A}\|q\|_{\mathcal A}.
\end{equation}
For each factor $1+a\cos(2\pi N_jx_{i_j})$, the sum of the absolute Fourier
coefficients is $1+a$; for $a\cos(2\pi N_kx_{i_k})$, it is $a$. Hence,
using the latter inequality of \eqref{eq:A norm}, we get
\begin{equation}\label{eq:wiener}
 \|g_k\|_{\mathcal A}\le a(1+a)^{k-1}.
\end{equation}

Let $e_1,e_2$ be the coordinate unit vectors.
By the Fourier expansion of $g_k$,
every frequency $\xi$ with $\widehat g_k(\xi)\ne0$ has the form
\[
 \xi=\sigma N_ke_{i_k}+\sum_{j<k}\varepsilon_jN_je_{i_j},
 \quad \sigma\in\{-1,1\},\quad \varepsilon_j\in\{-1,0,1\}.
\]
The new frequency therefore dominates in the $i_k$ coordinate, while
the other coordinate contains only earlier frequencies:
\begin{equation}\label{eq:frequency bounds}
 |\xi_{i_k}|\ge N_k-S_{k-1}>0,
 \quad
 |\xi_{3-i_k}|\le S_{k-1}.
\end{equation}

Define a periodic trigonometric polynomial $q_k$ by
\begin{equation}\label{eq:primitive}
 \widehat q_k(\xi)=
 \begin{cases}
  \widehat g_k(\xi)/(2\pi i\xi_{i_k}),&\xi_{i_k}\ne0,\\
  0,&\xi_{i_k}=0.
 \end{cases}
 \quad \partial_{i_k}q_k=g_k.
\end{equation}
Differentiation multiplies the coefficient at $\xi$ by
$2\pi i\xi_{i_k}$, while $\widehat g_k(\xi)=0$ when $\xi_{i_k}=0$.
Since $g_k$ is real, its coefficients are conjugate symmetric. The
definition therefore gives $\widehat q_k(-\xi)=\overline{\widehat q_k(\xi)}$,
so $q_k$ is real.

Set
\[
 w_k=
 \begin{cases}
  (0,2q_k),&k\text{ odd},\\
  (2q_k,0),&k\text{ even},
 \end{cases}
 \quad h_k=2\partial_{3-i_k}q_k.
\]
For a smooth $w\in C^{\infty}(\mathbb T^2;\R^2)$,
let $e(w)=(\nabla w+\nabla w^T)/2$ denote the
density of $Ew=e(w)\LL^2$. For odd $k$, the identity
$\partial_1q_k=g_k$ gives
\[
 e(w_k)_{1}^2=\tfrac12\partial_1(2q_k)=g_k,
 \quad e(w_k)_{2}^2=\partial_2(2q_k)=h_k,
 \quad e(w_k)_{1}^1=0.
\]
For even $k$, the coordinates are interchanged. Thus
\begin{equation}\label{eq:block-strain}
 Ew_k=
 \begin{cases}
  \begin{pmatrix}0&g_k\\g_k&h_k\end{pmatrix}\LL^2,
       &k\text{ odd},\\[2mm]
  \begin{pmatrix}h_k&g_k\\g_k&0\end{pmatrix}\LL^2,
       &k\text{ even}.
 \end{cases}
\end{equation}
Combining \eqref{eq:primitive}, \eqref{eq:frequency bounds}, and
\eqref{eq:wiener}, we get
\begin{align*}
 \|q_k\|_{\mathcal A}
 &\le\frac{\|g_k\|_{\mathcal A}}{2\pi(N_k-S_{k-1})}
 \le\frac{a(1+a)^{k-1}}{2\pi(N_k-S_{k-1})},\\
 \|\partial_{3-i_k}q_k\|_{\mathcal A}
 &=\sum_{\xi:\,\widehat g_k(\xi)\ne0}
       \frac{|\xi_{3-i_k}|}{|\xi_{i_k}|}\,|\widehat g_k(\xi)|
 \le\frac{a(1+a)^{k-1} S_{k-1}}{N_k-S_{k-1}}.
\end{align*}
Using the first inequality of \eqref{eq:A norm},
we obtain
\begin{align}
 \|w_k\|_{L^{\infty}(\mathbb T^2)}
 &\le\frac{a(1+a)^{k-1}}{\pi(N_k-S_{k-1})},\label{eq:w-bound}\\
 \|h_k\|_{L^{\infty}(\mathbb T^2)}
 &\le\frac{2a(1+a)^{k-1}S_{k-1}}{N_k-S_{k-1}}.\label{eq:h-bound}
\end{align}
By the growth condition \eqref{eq:growth},
\[
 \|w_k\|_{L^{\infty}(\mathbb T^2)}+\|h_k\|_{L^{\infty}(\mathbb T^2)} \le 2^{-k}.
\]

\emph{Step 2: The symmetric gradient of the limit.}\\
Define
\begin{equation}\label{eq:u-series}
 u(x_1,x_2)=(x_2,x_1)+\sum_{k=1}^{\infty}w_k(x_1,x_2),
 \quad
 H_1=\sum_{k\ {\rm even}}h_k,\quad
 H_2=\sum_{k\ {\rm odd}}h_k.
\end{equation}
All three series converge uniformly by Step 2. In particular,
$u\in C(\overline Q;\R^2)$, hence $u\in L^1(Q;\R^2)$, and $H_1,H_2$
are continuous periodic functions.

Consider the finite sums
\[
 u^{(m)}=(x_2,x_1)+\sum_{k=1}^m w_k,\quad
 H_1^{(m)}=\sum_{\substack{k\le m\\k\ {\rm even}}}h_k,\quad
 H_2^{(m)}=\sum_{\substack{k\le m\\k\ {\rm odd}}}h_k.
\]
The affine function $(x_2,x_1)$ has symmetric gradient
$\left(\begin{smallmatrix}0&1\\1&0\end{smallmatrix}\right)\LL^2$.
Since $F_0=1$, the off-diagonal terms telescope:
\[
 1+\sum_{k=1}^m g_k
 =1+\sum_{k=1}^m(F_k-F_{k-1})=F_m.
\]
Together with \eqref{eq:block-strain}, this gives
\begin{equation}\label{eq:finite-strain}
 Eu^{(m)}
 =\begin{pmatrix}H_1^{(m)}&F_m\\F_m&H_2^{(m)}\end{pmatrix}\LL^2.
\end{equation}

To pass to the distributional limit, fix $\varphi\in C_c^\infty(Q)$.
For the diagonal entries, integration by parts at the finite stages and
uniform convergence give
\begin{align*}
 -\int_Q u_1\partial_1\varphi\,dx
   &=-2\lim_{m\to\infty}\int_Q\sum_{\substack{k\le m\\k\ {\rm even}}} q_k\partial_1\varphi\,dx
   =\lim_{m\to\infty}\int_QH_1^{(m)}\varphi\,dx
   =\int_QH_1\varphi\,dx,\\
 -\int_Q u_2\partial_2\varphi \,dx
   &=-2\lim_{m\to\infty}\int_Q\sum_{\substack{k\le m\\k\ {\rm odd}}} q_k\partial_2\varphi\,dx
   =\lim_{m\to\infty}\int_QH_2^{(m)}\varphi\,dx
   =\int_QH_2\varphi\,dx.
\end{align*}
For the off-diagonal entry, we similarly get
\begin{align*}
 -\frac12\int_Q
       (u_2\partial_1\varphi+u_1\partial_2\varphi)\,dx
       &= -\frac12\lim_{m\to\infty}\int_Q\Big(
       \big(x_1+2\sum_{\substack{k\le m\\k\ {\rm odd}}} q_k\big)\partial_1\varphi+\big(x_2+2\sum_{\substack{k\le m\\k\ {\rm even}}} q_k\big)\partial_2\varphi\Big)\,dx\\
 &=\lim_{m\to\infty}\int_QF_m\varphi\,dx\\
   &=\int_Q\varphi\,d\mu.
\end{align*}
The last equality follows from Lemma~\ref{lem:riesz} after extending
$\varphi$ by zero to $\mathbb T^2$. The extension is smooth because
$\varphi$ vanishes near the boundary of $Q$.

These identities give
\begin{equation}\label{eq:Eu formula}
 Eu=\begin{pmatrix}H_1&0\\0&H_2\end{pmatrix}\LL^2
       +\begin{pmatrix}0&1\\1&0\end{pmatrix}\mu,
\end{equation}
which is \eqref{eq:main-strain}. The total variation is finite, since
\[
 |Eu|(Q)\le\|H_1\|_{L^1(Q)}+\|H_2\|_{L^1(Q)}+\sqrt2\,\mu(Q)<\infty.
\]
Thus $u\in\BD(Q)$. Note that this argument identifies the symmetric sum
$\partial_1u_2+\partial_2u_1$ as a measure without requiring either
individual cross derivative to be a measure.\\

\emph{Step 3: Separation from BV functions.}\\
Let $A,B\subset (0,1)$ be the sets of Lebesgue measure zero from Lemma~\ref{lem:riesz}, set
$F=A\times B$, and write
\[
 P=\begin{pmatrix}0&1\\1&0\end{pmatrix},\quad |P|=\sqrt2.
\]
The notation $\mu\mres F$
denotes the restriction of $\mu$ to $F$, that is,
\[
\mu \mres F(H):=\mu(F\cap H),\quad H\subset Q(0,1)\textrm{ is Borel}.
\]
Since $\LL^2(F)=0$ and $\mu(F)=1$, the absolutely continuous part of
$Eu$ vanishes on $F$, giving
\[
 Eu\mres F=P\mu.
\]
For $v=(v_1,v_2)\in\BV(Q;\R^2)$, apply
\eqref{eq:BV-null-product} to both components $v_j$. All four measures
$D_kv_j$ vanish on $F$, so $Dv\mres F=0$ and $Ev\mres F=0$.
Subtracting gives
\[
 E(u-v)\mres F=P\mu.
\]
Thus
\[
 |E(u-v)|(Q)\ge |E(u-v)|(F)=|P|\mu(F)=\sqrt2.
\]
This proves \eqref{eq:main-gap}.\\

\emph{Step 4: Higher dimensions.}\\
Fix $n\ge2$ and let $Q=(0,1)^n$, and
\[
 \widetilde u(x)=(u_1(x_1,x_2),u_2(x_1,x_2),0,\ldots,0),
\]
and
\[
 \widetilde\mu=\mu\otimes
      (\LL^{n-2}\mres(0,1)^{n-2}),\quad
 \widetilde F=A\times B\times(0,1)^{n-2}.
\]
Let
\[
H'(x)=\operatorname{diag}(H_1(x_1,x_2),H_2(x_1,x_2),0,\ldots,0),
\]
and define the $n\times n$ matrix $P$ so that
$(P)_{1}^2=(P)_{2}^1=1$ and all other entries are zero. 
The first two components of $\widetilde u$ are independent of the remaining coordinates, so
\eqref{eq:Eu formula} gives
\[
 E\widetilde u=H'\LL^n+P\widetilde\mu,\qquad 
\]
In
particular, $|P|=\sqrt2$, $\widetilde u\in\BD(Q)$, and
$\widetilde\mu$ is a probability measure carried by the Lebesgue-null
set $\widetilde F$.

Recall from \eqref{eq:projection} the orthogonal projections $\pi_j$.
We have
\[
 \LL^{n-1}\bigl(\pi_j(\widetilde F)\bigr)=0
 \quad\text{for every }j=1,\ldots,n.
\]
Thus from the one-dimensional slicing \eqref{eq:slice representation for total variation}, we get
$|Dw|(\widetilde F)=0$ for any $w\in\BV(Q,\R^n)$.

Next, choose a compact set $K\subset \widetilde F$ with
$\widetilde\mu(K)>0$, which exists by inner regularity. Choose
$\eta\in C_c^\infty(Q)$ with $0\le\eta\le1$ and $\eta=1$ on $K$.
Let $v=\eta\widetilde u$ on $Q$, extended by zero to $\R^n$.
We have $v\in C_c(Q)$.
For vectors $a,b\in\R^n$, we write
$a\odot b=(a\otimes b+b\otimes a)/2$. 
The distributional product rule gives
\[
 Ev=\eta E\widetilde u+
       (\widetilde u\odot\nabla\eta)\LL^n.
\]
These are finite measures, so $v\in\BD(\R^n)$.
On $K$ the
absolutely continuous terms vanish, leaving
\[
 Ev\mres K=P\widetilde\mu.
\]
%
Hence for every $w\in\BV(Q;\R^n)$,
\[
 |E(v-w)|(Q)\ge |E(v-w)|(K)=|Ev|(K)
 = \sqrt2\,\widetilde\mu(K)>0.
\]
The generalization from $Q$ to a general open set $\Omega\subset \R^n$ is straightforward. This completes
the proof of the theorem.
\end{proof}

\section*{Acknowledgements}
Q.-H. Nguyen was supported by the CAS Project for Young Scientists in
Basic Research (Grant No.~YSBR-031) and the NSFC (Grant Nos.~1251101538
and 12595282).

\section*{Use of AI}
The ideas and arguments in this paper were developed through discussions
between the authors and ChatGPT. ChatGPT also assisted with drafting
and identifying relevant literature. The authors subsequently rewrote
the manuscript in their own words. They independently rewrote all the
proofs and verified every argument. They take full responsibility for
the results, references, and final text.

\end{document}